\documentclass{amsart}
\usepackage{amsmath}
\usepackage{amsfonts}
\usepackage{enumerate}
\usepackage{amssymb}
\usepackage{amscd}
\usepackage[all]{xy}
\usepackage{bm}

\DeclareMathOperator{\gr}{gr}
\DeclareMathOperator{\Der}{Der}

\DeclareMathOperator{\ad}{ad}
\DeclareMathOperator{\SL}{SL}
\DeclareMathOperator{\GL}{GL}

\DeclareMathOperator{\End}{End}
\DeclareMathOperator{\Sym}{Sym}

\DeclareMathOperator{\adj}{adj}
\DeclareMathOperator{\Char}{char}
\DeclareMathOperator{\tr}{tr}

\begin{document}
\theoremstyle{plain}
\newtheorem{MainThm}{Theorem}
\renewcommand{\theMainThm}{\Alph{MainThm}}
\newtheorem{MainCor}{Corollary}
\renewcommand{\theMainCor}{\Alph{MainCor}}
\newtheorem*{trm}{Theorem}
\newtheorem*{lem}{Lemma}
\newtheorem*{prop}{Proposition}
\newtheorem*{defn}{Definition}
\newtheorem*{thm}{Theorem}
\newtheorem*{thmA}{Theorem A}
\newtheorem*{thmB}{Theorem B}
\newtheorem*{thmC}{Theorem C}
\newtheorem*{thmAA}{Theorem A'}
\newtheorem*{thmBB}{Theorem B'}
\newtheorem*{thmCC}{Theorem C'}
\newtheorem*{example}{Example}
\newtheorem*{cor}{Corollary}
\newtheorem*{conj}{Conjecture}
\newtheorem*{hyp}{Hypothesis}
\newtheorem*{thrm}{Theorem}
\newtheorem*{quest}{Question}
\theoremstyle{remark}
\newtheorem*{rem}{Remark}
\newtheorem*{rems}{Remarks}
\newtheorem*{notn}{Notation}
\newcommand{\FFp}{K}
\newcommand{\Fp}{\mathbb{F}_p}
\newcommand{\Fq}{\mathbb{F}_q}
\newcommand{\Zp}{\mathbb{Z}_p}
\newcommand{\Qp}{\mathbb{Q}_p}
\newcommand{\Kr}{\mathcal{K}}
\newcommand{\Rees}[1]{\widetilde{#1}}
\newcommand{\invlim}{\lim\limits_{\longleftarrow}}
\newcommand{\Md}[1]{\mathcal{M}(#1)}
\newcommand{\Pj}[1]{\mathcal{P}(#1)}
\newcommand{\G}[2]{\mathcal{G}_{#1}(#2)}
\newcommand{\fm}{\mathfrak m}
\newcommand{\Hom}{\operatorname{Hom}}
\newcommand{\Frac}{\operatorname{Frac}}
\newcommand{\Kdim}{\operatorname{Kdim}}
\newcommand{\GKdim}{\operatorname{GKdim}}
\newcommand{\Ext}{\operatorname{Ext}}
\newcommand{\injdim}{\operatorname{injdim}}
\let\le=\leqslant  \let\leq=\leqslant
\let\ge=\geqslant  \let\geq=\geqslant

\title[Reflexive ideals in Iwasawa algebras ]
{Nonexistence of reflexive ideals in \\
Iwasawa algebras of Chevalley type}

\author{K. Ardakov, F. Wei and J. J. Zhang}

\address{(K. Ardakov) School of Mathematical Sciences, University of Nottingham, University Park,
Nottingham, NG7 2RD, United Kingdom}

\email{Konstantin.Ardakov@nottingham.ac.uk}

\address{(F. Wei) Department of Applied Mathematics,
Beijing Institute of Technology, Beijing, 100081, P. R. China}

\email{daoshuo@bit.edu.cn}

\address{(J. J. Zhang) Department of Mathematics, Box 354350,
University of Washington, Seattle, Washington 98195, USA}

\email{zhang@math.washington.edu}

\begin{abstract}
Let $\Phi$ be a root system and let $\Phi(\Zp)$ be 
the standard Chevalley $\Zp$-Lie algebra associated to $\Phi$. For any integer $t\geq 1$,
let $G$ be the uniform pro-$p$ group corresponding to the powerful 
Lie algebra $p^t \Phi(\Zp)$ and suppose that $p\geq 5$. Then the Iwasawa algebra $\Omega_G$ 
has no nontrivial two-sided reflexive ideals. This was previously proved by the authors for the root system $A_1$. 
\end{abstract}



\maketitle

\setcounter{section}{-1}
\section{Introduction}
\label{xxsec0}

\subsection{Prime ideals in Iwasawa algebras}
\label{xxsec0.1} 
One of the main projects in the study of noncommutative 
Iwasawa algebras aims to understand the structure of two-sided
ideals in Iwasawa algebras $\Lambda_G$ and $\Omega_G$ for 
compact $p$-adic analytic groups $G$. A list of open questions 
in this project was posted in a survey paper by the first author 
and Brown \cite{AB}. Motivated by its connection to the Iwasawa 
theory of elliptic curves in arithmetic geometry it is particularly 
interesting to understand the prime ideals of $\Lambda_G$ 
when $G$ is an open subgroup of $\GL_2(\Zp)$. A reduction \cite{A} shows that this amounts to understanding the prime ideals of 
$\Omega_G$ when $G$ is an open subgroup of $\SL_2(\Zp)$. In a recent paper we introduced some machinery which allowed us to determine every prime ideal of $\Omega_G$ for any open 
torsionfree subgroup $G$ of $\SL_2(\Zp)$, see \cite[Theorem C]{AWZ}. In this paper the theory developed in \cite{AWZ} will be 
used to prove that under mild conditions on $p$, there are no non-zero reflexive ideals in 
$\Omega_G$ when $G$ is a uniform pro-$p$ group of Chevalley type. It follows from this that every two-sided reflexive 
ideal of $\Lambda_{G\times \Zp}$ is principal and centrally generated --- see \cite[Theorem 4.7]{A}.
\subsection{Definitions}
\label{xxsec0.2}
Throughout let $p$ be a fixed prime number. Let $\mathbb{Z}_p$ 
be the ring of $p$-adic integers and let ${\mathbb F}_p$ be the 
field ${\mathbb Z}/(p)$. Let $G$ be a compact $p$-adic analytic 
group. The \emph{Iwasawa algebra} of $G$ over ${\mathbb Z}_p$ 
(or the \emph{completed group algebra} of $G$ over ${\mathbb Z}_p$) 
is defined to be 
\[\Lambda_G : =\lim_{\longleftarrow} \mathbb{Z}_p[G/N],\]
where the inverse limit is taken over the open normal subgroups 
$N$ of $G$ \cite[p.443]{L}, \cite[p.155]{DDMS}. In this 
paper we use $R[G]$ for the group ring of $G$ over a ring
$R$. For any field $K$ of characteristic $p$, the 
\emph{Iwasawa algebra} of $G$ over $K$ (or the \emph{completed 
group algebra} of $G$ over $K$) is defined 
to be 
\[KG : =\lim_{\longleftarrow} K[G/N],\]
where the inverse limit is taken over the open normal subgroups 
$N$ of $G$. If $K=\Fp$, we write $\Omega_G$ for $KG$.

Let $A$ be any algebra and $I$ be a left ideal of $A$. We call $I$ 
is {\it reflexive} if the canonical map 
$$I\to \Hom_{A}(\Hom_{A}(I,A),A)$$
is an isomorphism. A reflexive right ideal is defined similarly. 
We will call a two-sided ideal $I$ of $A$ \emph{reflexive} if it 
is reflexive as a right and as a left ideal.

\subsection{Main results}
\label{xxsec0.3}
Let $\Phi$ be a root system, so that the
Dynkin diagram of any indecomposable component of $\Phi$ belongs to
$$\{A_n (n\geq 1), B_n (n\geq 2), C_n (n\geq 3), 
D_n (n\geq 4), E_6, E_7, E_8, F_4, G_2\}.$$ 
Let $\Phi(\Zp)$ denote the $\Zp$-Lie 
algebra constructed by using a Chevalley basis associated 
to $\Phi$. For any integer $t \geq 1$ (or $t\geq 2$ if $p=2$), the $\Zp$-Lie algebra $p^t \Phi(\Zp)$ is \emph{powerful}. By \cite[Theorem 9.10]{DDMS}
there is an isomorphism between the category of uniform
pro-$p$ groups and the category of powerful Lie algebras. The uniform pro-$p$ group corresponding to $p^t \Phi(\Zp)$
is called {\it of type $\Phi$}, or in general {\it of Chevalley type} 
without mentioning $\Phi$. In this case the Iwasawa algebra $\Omega_G$ 
is called {\it of type $\Phi$} (or {\it of Chevalley type} in general).

We say that $p$ is a \emph{nice prime} for $\Phi$ if $p\geq 5$ 
and if $p \nmid n+1$ when $\Phi$ has an indecomposable component of type $A_n$. Here
is our main theorem. 

\begin{thmA}
Let $\Phi$ be a root system and let $G$ be a uniform pro-$p$ group of type $\Phi$. If $p$ is a nice prime for 
$\Phi$, then the Frobenius pair $(\Omega_G, \Omega_{G^p})$ satisfies the derivation hypothesis. 
\end{thmA}

The undefined technical terms in Theorem A will be explained in Section 2. The following corollary was proved in \cite{AWZ}, assuming Theorem A. This paper fills in the missing step.

\begin{cor}\cite[Theorems A and B]{AWZ} Let $G$ be a torsionfree compact $p$-adic
analytic group whose $\Qp$-Lie algebra $\mathcal{L}(G)$ is split semisimple over $\Qp$. Suppose that $p$
is a nice prime for the root system $\Phi$ of $\mathcal{L}(G)$. Then
$\Omega_G$ has no non-trivial two-sided reflexive ideals. In particular, every non-zero normal element of $\Omega_G$ is a unit.
\end{cor}

It was asked in \cite[Question J]{AB} whether an Iwasawa 
algebra $\Omega_G$ of Chevalley type has any non-zero, non-maximal prime ideals. 
Theorem A says that it has no prime ideals of so-called 
homological height one and hence provides evidence for a 
negative answer. Combining this with a result of the first author
gives a complete answer to \cite[Question J]{AB} in the case when $\Phi = A_1$.

We conjecture that the hypothesis of $p$ being nice is 
superfluous. When $\Phi=A_1$ we gave a separate proof for $p=2$ in
\cite[Section 8]{AWZ} (see also Section 4), which shows the 
difficulty of dealing with non-nice primes. 

\subsection{An outline of the paper} In Section 1 we will give a treatment of some elementary material (linear algebra, derivations, Lie algebras) that will form an essential part of the proof of our main result. The reader may wish to skip this material on his first reading and return to it later as needed. Section 2 contains the definitions
of some key terms such as derivation hypothesis and Frobenius pair. The proof of Theorem A is given in Section 3.
Section 4 contains some remarks about the case when $\Phi=A_n$ and $p$ is not a nice prime.

\section{Preparatory results}
\label{xxsec1}

\subsection{A Vandermonde-type determinant}
\label{xxsec1.1}
Let $\{w_1,\ldots,w_m\}$ be a basis for an $m$-dimensional
$\Fp$-vector space $W$. Consider the symmetric algebra
\[B := \Sym(W) \cong \Fp[w_1,\ldots,w_m].\]
We are interested in the following matrix of Vandermonde type:
$$M(w_1,\ldots,w_m;d_1,\ldots,d_m):=
\begin{pmatrix} 
w_1^{p^{d_1}} & w_2^{p^{d_1}} & \cdots & w_m^{p^{d_1}} \\ 
w_1^{p^{d_2}} & w_2^{p^{d_2}} & \cdots & w_m^{p^{d_2}} \\ 
\vdots & \vdots & \cdots& \vdots \\ 
w_1^{p^{d_m}} & w_2^{p^{d_m}} & \cdots & w_m^{p^{d_m}}
\end{pmatrix}$$
where $\{d_1,\ldots,d_m\}$ is a sequence of non-negative 
integers. For simplicity we write
$$M(w_1,\ldots,w_m) := \begin{pmatrix} w_1 & w_2 & \cdots
& w_m \\ w_1^p & w_2^p & \cdots & w_m^p \\ \vdots & \vdots
& \cdots& \vdots \\ w_1^{p^{m-1}} & w_2^{p^{m-1}} & \cdots
& w_m^{p^{m-1}}\end{pmatrix}.$$

Let $\mathbb{P}(W)$ be the set of all one-dimensional subspaces of
$W$. For each $l \in \mathbb{P}(W)$ we fix a choice of generator
$w_l \in l$ so that $l = \langle w_l\rangle$. Define
$\Delta(W)$ to be the product $\prod_{l \in \mathbb{P}(W)} w_l$.
 
\begin{lem} 
\begin{enumerate}
\item
Let $\{d_1,\cdots,d_m\}$ be a sequence of non-negative integers. 
Then $\Delta(W)$ divides $\det M(w_1,\ldots,w_m;d_1,\ldots,d_m)$. 
\item
There exists $\lambda \in \Fp^{\times}$ such that
$\det M(w_1,\ldots,w_m) = \lambda \cdot \Delta(W)$.
\end{enumerate}
\end{lem}
\begin{proof}
(1) Let $w$ be a non-zero element of $W$; we will show that
\[\det M\in wB.\]
where $M=M(w_1,\ldots,w_m;d_1,\ldots,d_m)$.
Write $w = a_1w_1 + \ldots + a_mw_m$ for some $a_i \in \Fp$, not all
zero. Without loss of generality $a_1 = -1$. Consider the canonical
ring homomorphism $\pi : \Sym(W) \to \Sym(W / \langle w\rangle)$;
this has kernel exactly $wB$. Let $u_i = \pi(w_i)$; then
\[\pi(\det M) = \det \begin{pmatrix} 
a_2u_2^{p^{d_1}}+ \cdots + a_mu_m^{p^{d_1}} 
& u_2^{p^{d_1}} & \cdots & u_m^{p^{d_1}} \\ 
a_2u_2^{p^{d_2}} + \cdots +a_mu_m^{p^{d_2}} 
& u_2^{p^{d_2}} & \cdots & u_m^{p^{d_2}}\\ 
\vdots & \vdots & \cdots& \vdots \\ 
a_2u_2^{p^{d_m}}+\cdots+a_mu_m^{p^{d_m}}
& u_2^{p^{d_m}} & \cdots & u_m^{p^{d_m}}
\end{pmatrix}\] 
which is zero because the first column is a linear combination of 
the others. Hence $\det M \in wB$ as claimed.

Now if $l\neq l'$ are two distinct lines then $w_l$ and $w_{l'}$ are
coprime in $B$. Hence $\Delta(W)=\prod_{l \in \mathbb{P}(W)} w_l$ divides
$\det M$. 

It is well known that $\det M$ is non-zero if and only if 
$\{d_1,\ldots,d_m\}$ are distinct.

(2) Both expressions are polynomials of degree precisely
\[1 + p + p^2 + \cdots + p^{m-1} = |\mathbb{P}(W)| = \frac{p^m-1}{p-1}\]
and the result follows.
\end{proof}

\subsection{The adjugate matrix}
\label{xxsec1.2} 
Later on we will be interested in coming as close as
possible to inverting the matrix $M(w_1,\ldots,w_m)$. Recall
\emph{Cramer's rule}: this says that if $A$ is any square $m\times
m$ matrix then
\[\adj(A) \cdot A = A \cdot \adj(A) = \det(A)\cdot I_m\]
where $I_m$ is the identity matrix and $\adj(A)$ is the
\emph{adjugate matrix}, defined as follows:
\[\adj(A)_{ij} = (-1)^{i+j} \det C_{ji}\]
where $C_{ij}$ is the matrix $A$ with the $i^{\rm{th}}$ row and
$j^{\rm{th}}$ column removed. 

We will use the following standard piece of notation. Given a list $(x_1,\ldots,x_n)$ consisting of $n$ elements, $(x_1,\ldots,\widehat{x_j},\ldots,x_n)$ denotes the list consisting of $n-1$ elements, where $x_j$ has been omitted. Thus $M(w_1,\ldots,\widehat{w_j},\ldots,w_m)$ is equal to the $(m-1)\times
(m-1)$-matrix defined in $\S \ref{xxsec1.1}$ with 
$\{d_1,\cdots,d_{m-1}\}=\{0,1,\cdots,m-2\}$. Lemma
\ref{xxsec1.1} implies the following

\begin{prop} 
Retain the notation of $\S \ref{xxsec1.1}$ and let
$A = M(w_1,\ldots,w_m)$. Then for any $j = 1,\ldots,m$,
$\det M(w_1,\ldots,\widehat{w_j},\ldots,w_m)$ divides each
entry in the $j^{\rm{th}}$ row of $\adj(A)$ in $B$.
\end{prop}

\begin{proof}
We need to show that for all $i = 1,\ldots,m$,
\[\det M(w_1,\ldots,\widehat{w_j},\ldots,w_m) | \det C_{ij}.\]
But $C_{ij}=M(w_1,\ldots,\widehat{w_j},\ldots,w_m;
0,\ldots, \widehat{i-1},\ldots, m-1)$. The assertion follows from
Lemma \ref{xxsec1.1}(1).
\end{proof}

For each $j=1,\ldots,m$, let $W_j$ be the subspace $\langle
w_1,\ldots,\widehat{w_j},\ldots,w_m\rangle$ of $W$. Define
\[\Delta_j := \prod_{l \in \mathbb{P}(W) \setminus \mathbb{P}(W_j)}w_l \in B.\]
By Lemma \ref{xxsec1.1}(2), we see that for some $\lambda_j \in
\Fp^\times$,
\[ \Delta_j = \lambda_j \cdot \frac{\det M(w_1,\ldots,w_m)}
{\det M(w_1,\ldots,\widehat{w_j},\ldots,w_m)}. \]

\begin{cor} 
Let $D$ be the diagonal $m\times m$ matrix defined by
$D_{ij} = \delta_{ij}\Delta_j$. Then there exists $U \in M_m(B)$ 
such that $U \cdot A = D$.
\end{cor}

\begin{proof} Let $E$ be the diagonal $m\times m$ matrix defined by 
\[\lambda_j E_{ij} = \delta_{ij} \det M(w_1,\ldots,\widehat{w_j},\ldots,w_m).\]
By the proposition, there exists $U \in M_m(B)$ such that $E\cdot U = \adj(A).$
Hence $U \cdot A = E^{-1} \cdot \adj(A) \cdot A = E^{-1} \det A = D$, as required.
\end{proof}

\subsection{Derivations} 
\label{xxsec1.3}
Now let $V$ be a finite dimensional vector space over a field $K$ 
(soon we will assume that $K=\Fp$). Consider the set $\mathfrak{D}$ 
of all derivations of $B := \Sym_K(V)$. Note that any $f \in V^\ast
:=\Hom_K(V,K)$ in the dual space of $V$ gives rise to a derivation, 
which we again denote by $f$, defined by the rule
\[f(v_1\cdots v_k) = \sum_{j=1}^k v_1\cdots f(v_j) \cdots v_k\]
for all $v_1,\ldots,v_k\in V$. Next, $\mathfrak{D}$ is naturally a
left $B$-module, with the action\ given by
\[(b\cdot d)(x) = bd(x)\]
for all $b,x\in B$ and $d\in \mathfrak{D}$. This gives us a
$K$-linear map $\psi : B\otimes V^{\ast} \to \mathfrak{D}$, defined
by $\psi(b\otimes f) = b\cdot f$. The following Lemma is well-known.

\begin{lem}
Let $\psi$ be defined as above. Then $\psi$ is a $B$-module 
isomorphism.
\end{lem}

Now we assume that $K=\Fp$. Then $x \mapsto x^{p^r}$ 
is an $\Fp$-linear endomorphism of $B$. Hence it extends to an 
$\Fp$-linear endomorphism, denoted by $(-)^{[p^r]}$, of 
$B \otimes V^{\ast}$ that is determined by
\[(b\otimes f)^{[p^r]} = b^{p^r}\otimes f.\]

\begin{defn} For any $d\in\mathfrak{D}$ and $r\geq 0$, let
$d^{[p^r]} = \psi(\psi^{-1}(d)^{[p^r]})$ be the
corresponding derivation.
\end{defn}

Thus $d^{[p^r]}$ is the derivation of $B$ determined by the rule
\[d^{[p^r]}(v) = d(v)^{p^r}\]
for all $v\in V$. We will henceforth identify $B\otimes V^{\ast}$ with $\mathfrak{D}$
using $\psi$. 

\subsection{A certain module of derivations}
\label{xxsec1.4}
The space $\End(V)$ can be canonically identified with $V \otimes
V^{\ast}$. Since $V$ is contained in $B$, we will identify 
$\End(V)$ with $V \otimes V^{\ast} \subseteq \mathfrak{D}$.

As in Section \ref{xxsec1.1}, for each $l \in \mathbb{P}(V)$ choose
some $v_l \in V$ such that $l = \langle v_l\rangle$. If $\varphi \in
\End(V)$ then $\varphi^\ast \in \End(V^\ast)$ is the dual map to $\varphi$ defined by
\[\varphi^\ast(g) = g\circ \varphi\]
for all $g\in V^\ast$.

\begin{prop} Let $\varphi \in \End(V)$ and $s \geq 0$ be given.
Consider the $B$-submodule
\[\mathcal{E}_s := \sum_{r\geq s} B \cdot \varphi^{[p^r]}\]
of $\mathfrak{D}$, and let $g \in V^\ast$. Then
\[\left(\prod_{l \in \mathbb{P}(\varphi(V))\setminus 
\mathbb{P}(\ker g)}
v_l^{p^s}\right) \cdot \varphi^{\ast}(g) \in \mathcal{E}_s.\]
\end{prop}

\begin{proof} Let $W = \varphi(V)$ and write $m = \dim W$ and 
$n = \dim V$. Consider the annihilator $(\ker\varphi)^{\perp}$ 
of $\ker\varphi$ in $V^\ast$. This clearly contains 
$\varphi^\ast(V^\ast)$ and is hence equal to it because both 
spaces have dimension $m$.

There is nothing to prove if $\varphi^*(g)=0$. Otherwise let 
$\{f_1,\ldots,f_m\}$ be a basis for $(\ker\varphi)^{\perp}$ such
that $f_m = \varphi^\ast(g)$, and extend it to a basis
$\{f_1,\ldots,f_n\}$ for $V^\ast$. Let $\{v_1,\ldots,v_n\}$ be the
dual basis for $V$, so that
\[f_i(v_j) = \delta_{ij}\quad {\text{for all $i,j$}}.\]
Then $\{v_{m+1},\ldots,v_n\}$ is a basis for $\ker \varphi$ and
$\{w_1,\ldots,w_m\}$ is a basis for $W = \varphi(V)$, where $w_i :=
\varphi(v_i)$ for $i=1,\ldots,m$.

Inside $\mathfrak{D}$ we have $\varphi = \sum_{i=1}^m w_i\cdot f_i$
by construction, so
\begin{equation}
\label{1.4.1}
\varphi^{[p^r]} = \sum_{i=1}^m w_i^{p^r}\cdot f_i
\tag{1.4.1}
\end{equation}
for all $r\geq 0$.

Consider the vector space $W^{[p^s]} = \langle
w_1^{p^s},\ldots,w_m^{p^s}\rangle$ and let $A =
M(w_1^{p^s},\ldots,w_m^{p^s})$ be the matrix appearing in
Section \ref{xxsec1.1}. Now $\mathfrak{D}$ is a left $B$-module, so
$\mathfrak{D}^m:=
\begin{pmatrix}\mathfrak{D}\\
\vdots\\\mathfrak{D}\end{pmatrix}$ is a left $M_m(B)$-module. 
Let $\mathbf{e} \in \mathfrak{D}^m$ be the column vector whose 
$r^{\rm{th}}$ entry is the
derivation $\varphi^{[p^{r + s - 1}]}$, and let $\mathbf{f} \in
\mathfrak{D}^m$ be the column vector whose $r^{\rm{th}}$ entry is the
derivation $f_r$, for each $r =1,\ldots,m$. Then we can rewrite the
equations \eqref{1.4.1} for $r = s,s+1,\ldots,s+m-1$ as
\[ A \cdot \mathbf{f} = \mathbf{e} \in \mathcal{E}_s^m\]
inside $\mathfrak{D}^m$. By Corollary \ref{xxsec1.2} we can find $U\in
M_m(B)$ such that $U \cdot A = D$ is a diagonal matrix whose
$j^{\rm{th}}$ entry is
\[\Delta_j = \prod_{l \in \mathbb{P}(W) \setminus 
\mathbb{P}(W_j)}v_l^{p^s}\]
Here $W_j = \langle w_1,\ldots,\widehat{w_j},\ldots,w_m \rangle$,
for all $j =1,\ldots,m$. Hence
\[D\cdot \mathbf{f} = U\cdot A\cdot \mathbf{f} 
= U\cdot \mathbf{e} \in \mathcal{E}_s^m,\]
so in particular $D_m \cdot \varphi^\ast(g) = D_m\cdot f_m \in
\mathcal{E}_s$. Now $g(w_i) = (g\circ\varphi)(v_i) = f_m(v_i) =
\delta_{mi}$ for all $i=1,\ldots, m$, so
\[W_m = W \cap \ker g.\]
The result follows.
\end{proof}

\subsection{$\mathfrak{g}$-modules}
\label{xxsec1.5}
Let $\mathfrak{g}$ be a finite dimensional Lie algebra a base 
field $k$. By a \emph{$\mathfrak{g}$-module} we mean a left
$U(\mathfrak{g})$-module $V$, where $U(\mathfrak{g})$ is the
universal enveloping algebra of $\mathfrak{g}$. To give $V$ the
structure of a $\mathfrak{g}$-module is the same thing as to 
give a Lie algebra homomorphism
\[\rho : \mathfrak{g} \to \mathfrak{gl}(V)\]
where $\mathfrak{gl}(V) = \End(V)$ is the Lie algebra of all 
linear endomorphisms of $V$ under the commutator bracket.

If $V$ is a $\mathfrak{g}$-module then so is the dual space
$V^{\ast}$, by the rule
\[(x \cdot f)(v) = -f(x \cdot v)\]
for all $x\in\mathfrak{g}$, $f\in V^\ast$ and $v\in V$. Note that
$\rho : \mathfrak{g} \to \mathfrak{gl}(V)$ and $\rho^* :
\mathfrak{g} \to \mathfrak{gl}(V^\ast)$ are the corresponding
representations then
\[\rho^\ast(x) = - \rho(x)^\ast\]
for all $x\in\mathfrak{g}$. Here as in Section \ref{xxsec1.4},
$\rho(x)^\ast$ denotes the dual map to $\rho(x) : V\to V$.

\subsection{Invariant bilinear forms}
\label{xxsec1.6}
Let $V$ be a $\mathfrak{g}$-module. Recall that a
\emph{$\mathfrak{g}$-invariant form} on $V$ is a bilinear form
\[( \; ,\;  ) : V \times V \to k\]
such that $(x \cdot v, w) = -(v, x \cdot w)$ for all
$x\in\mathfrak{g}$ and $v, w\in V$. Such a form determines a
homomorphism of $\mathfrak{g}$-modules $\beta : V \to V^\ast$ 
via the rule $\beta(v)(w) = (v, w)$, and conversely, a
$\mathfrak{g}$-module homomorphism $V \to V^\ast$ defines a
$\mathfrak{g}$-invariant form on $V$. Note that the form 
$(\; ,\; )$ is \emph{non-degenerate} if and only if the 
associated homomorphism $\beta$ is an isomorphism.

\subsection{The adjoint representation}
\label{xxsec1.7}
Now consider $V = \mathfrak{g}$ as a $\mathfrak{g}$-module 
via $x \cdot y = [x, y]$ for all $x, y\in\mathfrak{g}$. The 
following elementary result will be very useful later on.

\begin{lem} 
Suppose that $\mathfrak{g}$ has a $\mathfrak{g}$-invariant
bilinear form $(\; ,\; )$, and let $\beta$ be the associated
homomorphism. Then for all $x, y\in\mathfrak{g}$,
\begin{enumerate}[{(}a{)}]
\item $x\cdot \beta(y) = y\cdot \beta(-x)$, and
\item $[x, \mathfrak{g}] \subseteq \ker \beta(x)$.
\end{enumerate}
\end{lem}

\begin{proof} (a) $x\cdot \beta(y) = \beta([x,y]) = \beta([y,-x]) 
= y\cdot \beta(-x)$.

(b) $\beta(x)([x,\mathfrak{g}]) = (x,[x,\mathfrak{g}]) =
([-x,x],\mathfrak{g}) = 0.$
\end{proof}

\subsection{The Killing form}
\label{xxsec1.8}
Recall that the \emph{Killing form} on $\mathfrak{g}$ is defined by
the rule
\[\mathcal{K}(x,y) = \tr(\ad(x)\ad(y))\]
for all $x,y\in\mathfrak{g}$. This is always an example of a
$\mathfrak{g}$-invariant bilinear form on $\mathfrak{g}$. If $\Char
k = 0$ then Cartan's Criterion states that $\mathfrak{g}$ is
semisimple if and only if the Killing form is non-degenerate.
However in positive characteristic it may happen that $\mathfrak{g}$
is simple but its Killing form is zero. This can happen even when
$\mathfrak{g}$ is of "classical type", meaning that it is a
Chevalley Lie algebra over $k$. There is a way around this problem
--- see the proof of Theorem \ref{xxsec3.4}.

\section{Frobenius pairs and the derivation hypothesis}
\label{xxsec2}
In this section we review a minimal amount of material from 
\cite{AWZ} that is most relevant for this paper. In particular 
we will recall the derivation hypothesis which plays a key 
role in \cite{AWZ}. Together with the main theorem of this
paper, the theory in \cite{AWZ} leads to a proof of the 
structure theorem for reflexive ideals in a class of 
Iwasawa algebras. 

\subsection{Frobenius pairs}
\label{xxsec2.1}
We go back to an arbitrary base field $K$ of characteristic $p$. Let $B$ be a commutative 
$K$-algebra; for example $B$ could be the polynomial algebra
$$B=\gr KG=\Sym(V\otimes K)$$
for some finite dimensional $\Fp$-vector space $V$.
The Frobenius map $x \mapsto x^p$ is a ring endomorphism 
of $B$ and gives an isomorphism of $B$ onto its image
\[B^{[p]} := \{b^p : b \in B\} \]
in $B$ provided that $B$ is reduced. Any derivation $d : B \to B$ 
is $B^{[p]}$-linear because
\[d(a^pb) = a^pd(b) + pa^{p-1}d(a)b = a^pd(b)\]
for all $a,b \in B$.

Let $t$ be a positive integer. Whenever $\{y_1,\ldots,y_t\}$ is a
$t$-tuple of elements of $B$ and $\alpha =
(\alpha_1,\ldots,\alpha_t)$ is a $t$-tuple of nonnegative integers,
we define
\[{\mathbf{y}}^{\alpha} =y_1^{\alpha_1}\cdots y_t^{\alpha_t}.\]
Let $[p-1]$ denote the set $\{0,1,\ldots,p-1\}$ and let $[p-1]^t$ be
the product of $t$ copies of $[p-1]$.

\begin{defn}\cite[Definition 2.2]{AWZ}
Let $A$ be a complete filtered $\FFp$-algebra and let $A_1$ be
a subalgebra of $A$. We always view $A_1$ as a filtered subalgebra
of $A$, equipped with the subspace filtration $F_nA_1 := F_nA \cap
A_1$. We say that $(A,A_1)$ is a \emph{Frobenius pair} if the
following axioms are satisfied:
\begin{enumerate}[{(} i {)}]
\item 
$A_1$ is closed in $A$,
\item 
$\gr A$ is a commutative noetherian domain, and we write 
$B=\gr A$,
\item 
the image $B_1$ of $\gr A_1$ in $B$
satisfies $B^{[p]} \subseteq B_1$, and
\item 
there exist homogeneous elements $y_1,\ldots,y_t \in B$ 
such that
\[B = \bigoplus_{\alpha \in [p-1]^t}  
B_1 \mathbf{y}^{\alpha}.\]
\end{enumerate}
\end{defn}

\subsection{Derivations on $B$}
\label{xxsec2.2} 
Let $B_1 \subseteq B$ be commutative rings of characteristic $p$, 
such that $B^{[p]} \subseteq B_1$ and
\[B = \bigoplus_{\alpha\in[p-1]^t} B_1\mathbf{y}^\alpha\]
for some elements $y_1,\ldots, y_t$ of $B$.

Fix $j = 1,\ldots,t$ and let $\epsilon_j$ denote the $t$-tuple of
integers having a $1$ in the $j$-th position and zeros elsewhere. We
define a $B_1$-linear map $\partial_j : B \to
B$ by setting
\[\partial_j\left(\sum_{\alpha\in[p-1]^t} u_\alpha
\mathbf{y}^\alpha \right) := \sum_{\stackrel{\alpha\in[p-1]^t}
{\alpha_j > 0}} \alpha_ju_\alpha \mathbf{y}^{\alpha - \epsilon_j}.\]

Let $\mathcal{D} := \Der_{B_1}(B)$ denote the
set of all $B_1$-linear derivations of $B$. 
An ideal $I$ of $B$ is called {\it $\mathcal{D}$-stable}
if $\mathcal{D}\cdot I\subseteq I$.

\begin{prop} \cite[Proposition 2.4]{AWZ}
\begin{enumerate}[{(}a{)}]
\item 
The map $\partial_j$ is a $B_1$-linear derivation of $B$ for each $j$.
\item 
$\mathcal{D} = \bigoplus_{j=1}^t B \partial_j.$
\item 
For any $x \in B$, $\mathcal{D}(x) = 0$ if and only if $x \in B_1$.
\item 
An ideal $I \subseteq B$ is $\mathcal{D}$-stable
if and only if it is controlled by $B_1$:
\[I = (I\cap B_1)B.\]
\end{enumerate}
\end{prop}

If $K=\Fp$ and $B=\Sym(V)$ then  $\mathfrak{D}=
\mathcal{D}$. Part (a) of the above is similar to 
Lemma \ref{xxsec1.3}.

\subsection{Inducing derivations on $\gr A$}
\label{xxsec2.3}
 Let $A$ be a filtered ring with associated graded
ring $B$ and let $a\in A$. Suppose that there is an
integer $n\geq 0$ such that
\[[a,F_kA] \subseteq F_{k-n}A\]
for all $k\in\mathbb{Z}$. This induces linear maps
\[\begin{array}{cccc}
\{a, - \}_n :& \frac{F_kA}{F_{k-1}A} &\to & 
\frac{F_{k-n}A}{F_{k-n-1}A} \\
\quad &\quad \\
& b + F_{k-1}A &\mapsto &[a,b] + F_{k-n-1}A
\end{array}
\]
for each $k\in\mathbb{Z}$ which piece together to give a graded
derivation
\[\{a, - \}_n : B \to B.\]

\begin{defn}\cite[Definition 3.2]{AWZ}
A \emph{source of derivations} for a Frobenius pair
$(A,A_1)$ is a subset $\mathbf{a} = \{a_0,a_1,a_2,\ldots\}$ of $A$
such that there exist functions $\theta, \theta_1:\mathbf{a}
\to\mathbb{N}$ satisfying the following conditions:
\begin{enumerate}[{(}i{)}]
\item
$[a_r, F_kA] \subseteq F_{k - \theta(a_r)}A$ for all $r\geq 0$ 
and all $k\in\mathbb{Z}$
\item
$[a_r, F_kA_1] \subseteq F_{k - \theta_1(a_r)}A$ for all $r\geq 0$
and all $k\in\mathbb{Z}$,
\item $\theta_1(a_r) - \theta(a_r) \to \infty$ as $r \to \infty$.
\end{enumerate}
Let $\mathcal{S}(A,A_1)$ denote the set of all sources of
derivations for $(A,A_1)$.
\end{defn}

\subsection{The derivation hypothesis}
\label{xxsec2.4} 
Let $\mathbf{a}$ be a source of derivations for a Frobenius pair
$(A, A_1)$ and $I$ be a graded ideal of $B$. We say that
the homogeneous element $Y$ of $B$ lies in the 
\emph{$\mathbf{a}$-closure} of $I$ if $\{a_r,Y\}_{\theta(a_r)}$ 
lies in $I$ for all $r \gg 0$.

Each source of derivations $\mathbf{a}$ gives rise to a sequence of
derivations $\{a_r,-\}_{\theta(a_r)}$ of $B$, and some or
all of these could well be zero. To ensure that we get an
interesting supply of derivations of $B$, we now
introduce a condition which holds for Iwasawa algebras of only
rather special uniform pro-$p$ groups.

Recall that $\mathcal{D}$ denotes the set of all
$B_1$-linear derivations of $B$ and
$\mathcal{S}(A,A_1)$ denotes the set of all sources of derivations
for $(A, A_1)$. The derivation hypothesis is really concerned with the
action of the derivations induced by $\mathcal{S}(A, A_1)$ on the
graded ring $B$.

\begin{defn}\cite[Definition 3.5]{AWZ}
Let $(A, A_1)$ be a Frobenius pair. We say that $(A,A_1)$ satisfies 
the {\sf derivation hypothesis} if for all homogeneous $X,Y \in B$ such that $Y$ lies in the
$\mathbf{a}$-closure of $XB$ for all $\mathbf{a}\in\mathcal{S}(A,A_1)$, we must have $\mathcal{D}(Y)
\subseteq XB$.
\end{defn}

Using this hypothesis, it is possible to prove the following 
control theorem for reflexive ideals:

\begin{trm}\cite[Theorem 5.3]{AWZ} Let $(A,A_1)$ be a Frobenius pair satisfying 
the derivation hypothesis, such that $B$ and $B_1$ are UFDs. Let $I$ be a reflexive two-sided
ideal of $A$. Then $I\cap A_1$ is a reflexive two-sided
ideal of $A_1$ and $I$ is controlled by $A_1$:
\[I = (I\cap A_1)\cdot A.\]
\end{trm}

This is the main technical result of \cite{AWZ}, which eventually implies Corollary \ref{xxsec0.3}.

\section{Proof of the main result}
\label{xxsec3}

\subsection{Normalizers of powerful Lie algebras}
\label{xxsec3.1}
Recall from \cite[\S 9.4]{DDMS} that a $\Zp$-Lie algebra $L$ is 
said to be \emph{powerful} if $L$ is free of finite rank as a 
module over $\Zp$ and $[L,L] \subseteq p^\epsilon L$, where 
$$\epsilon := \left\{ \begin{array}{l}2 \quad \mbox{ if }
\quad p=2 \\ 1 \quad \mbox{ otherwise.} \end{array} \right.$$ 
Let $L$ be a powerful $\Zp$-Lie algebra and
let $N = \{x \in \Qp L : [x,L] \subseteq L\}$ be the normalizer of
$L$ inside $\Qp L$ --- this is a $\Zp$-subalgebra of $\Qp L$ that
contains $L$ as an ideal. Note that $N$ is just the inverse image of
$\End_{\Zp}(L)$ under the homomorphism
\[\ad : \Qp L \to \End_{\Qp}(\Qp L),\]
so $N$ contains the center $Z(\Qp L)$ of $\Qp L$ and $N / Z(\Qp L)$ 
is a finitely generated $\Zp$-module. Hence
\[\mathfrak{g} := N/pN\]
is a finite dimensional $\Fp$-Lie algebra. Define
\[V := L/pL.\]
Then $V$ is naturally a $\mathfrak{g}$-module via the rule
\[(x + pN)\cdot (y + pL) = [x,y] + pL\]
for all $x\in N$ and $y\in L$. Let $\rho : \mathfrak{g} \to \End
(V)$ be the associated homomorphism.

\begin{lem}
Let $x\in N\backslash pN$ and $k \geq \epsilon$ be such
that $u = p^kx \in L$. Then
\begin{enumerate}[{(}a{)}]
\item $[u,L]\subseteq p^kL$
\item $[u,L]\nsubseteq p^{k+1}L$, and
\item $[u,pL]\subseteq p^{k+1}L$.
\end{enumerate}
\end{lem}

\begin{proof}
The first and the last assertions are clear. If $[u,L] \subseteq p^{k+1}L$
then $[x,L]\subseteq pL$ and so $p^{-1}x \in N$. But this forces $x \in pN$,
which we have assumed not to be the case.
\end{proof}

\subsection{Derivations for Iwasawa algebras}
\label{xxsec3.2}
By \cite[Theorem 9.10]{DDMS} there is a natural assignment
$$G\mapsto \log (G),\qquad L\mapsto \exp(L)$$
which determines an equivalence between the category of uniform 
pro-$p$ groups and the category of powerful $\Zp$-Lie algebras.

Now let $G = \exp(L)$ be the uniform pro-$p$ group
corresponding to our powerful Lie algebra $L$, and let $K$ denote an
arbitrary field of characteristic $p$. By \cite[Proposition
6.6]{AWZ}, $(KG,KG^p)$ is a Frobenius pair, and by \cite[Lemma
6.2(d) and Proposition 6.4]{AWZ}, there is a canonical isomorphism
\[\Sym(V\otimes K) \stackrel{\cong}{\longrightarrow} \gr KG.\]
Recall that $\rho$ is the map $\mathfrak{g}\to \End(V)\subseteq 
\End(V\otimes K)$ defined in Section \ref{xxsec3.1}.

\begin{prop} 
Let $x \in N\backslash pN$ and let $k\geq 1$ be such that
$p^kx \in L$. Let $a = \exp(p^kx) \in G$. Then
\[\{a,-\}_{p^k-1} = \rho(x + pN)^{[p^k]}\]
as derivations of $\Sym(V\otimes K)$.
\end{prop}

\begin{proof} 
This is a rephrasing of \cite[Theorem 6.8]{AWZ},
using Lemma \ref{xxsec3.1}.
\end{proof}

\subsection{Verifying the derivation hypothesis}
\label{xxsec3.3}
We start with a powerful $\Zp$-Lie algebra $L$ and 
define $\mathfrak{g}$ and $V$ as in Section \ref{xxsec3.1}.
Let $G=\exp(L)$. We say $L$ \emph{satisfies hypothesis $(L*)$} if 
the following hold:

\begin{enumerate}
\item[(L0)] 
there exists a $\mathfrak{g}$-module isomorphism 
$\zeta : \mathfrak{g} \to V$,
\item[(L1)] 
$\sum_{\beta} \mathfrak{g}\cdot \beta(\mathfrak{g})=
\mathfrak{g}^*$ where the sum runs over all possible
$\mathfrak{g}$-module homomorphisms $\beta: 
\mathfrak{g}\to \mathfrak{g}^*$.
\end{enumerate}

Since $\mathfrak{g}\cdot \beta(\mathfrak{g})=
\beta([\mathfrak{g},\mathfrak{g}])$, condition
(L1) is equivalent to 
$\sum_{\beta}\beta([\mathfrak{g},\mathfrak{g}])=
\mathfrak{g}^*$.
Clearly, the following two conditions imply (L1):
\begin{itemize}
\item
$\mathfrak{g}$ admits a non-degenerate 
$\mathfrak{g}$-invariant bilinear form $(\; ,\; )$, and 
\item
$\mathfrak{g}$ is \emph{perfect} : 
$[\mathfrak{g},\mathfrak{g}] = \mathfrak{g}$.
\end{itemize}

\begin{thm}
Let $L$ be a powerful Lie algebra satisfying hypothesis
$(L*)$ and let $G=\exp(L)$. Then the Frobenius pair 
$(KG,KG^p)$ satisfies the derivation hypothesis.
\end{thm}

\begin{proof}
Let $X, Y$ be homogeneous elements of $B = \gr KG$ such
that $Y$ lies in the $\mathbf{a}$-closure of $XB$ for all
$\mathbf{a} \in \mathcal{S}(KG, KG^p)$. Let $x\in \mathfrak{g}$ be a
non-zero element, and suppose that $x = x' + pN$ for some $x' \in N \backslash pN$. Let $k \geq \epsilon$ be such that $p^kx' \in L$. Then
\[(\exp(p^kx'), \exp(p^{k+1}x'), \exp(p^{k+2}x'),\ldots)\]
is a source of derivations for $(KG, KG^p)$, by \cite[Corollary
6.7]{AWZ}. Hence there exists a large integer $s_x \geq k$, such
that
\[\{\exp(p^rx'), Y\}_{p^r - 1} \in XB\]
for all $r\geq s_x$. Hence by Proposition \ref{xxsec3.2} we see that
\[\rho(x)^{[p^r]}(Y) \in XB\]
for all $r \geq s_x$. Since $\mathfrak{g}$ is finite, if we set $s
:= \max \{s_x : x \in \mathfrak{g}\backslash 0\}$ then
\[\rho(x)^{[p^r]}(Y) \in XB\]
for all $r \geq s$ and all $x\in \mathfrak{g}$.

Let us identify $\Sym(\mathfrak{g} \otimes K)$ with $\Sym(V\otimes
K)$ using the isomorphism $\zeta$ in (L0). Then
\[\ad(x)^{[p^r]}(Y) \in XB\]
for all $r \geq s$ and all $x\in \mathfrak{g}$. 

Let $(\;,\;)$ be any $\mathfrak{g}$-invariant bilinear
form on $\mathfrak{g}$, and let $\beta :
\mathfrak{g} \to \mathfrak{g}^\ast$ be the associated homomorphism.

Fix $x,y\in \mathfrak{g}$, let $\varphi := \ad(x) \in
\End(\mathfrak{g})$ and let $g = \beta(y) \in \mathfrak{g}^\ast$.
Then $\varphi^\ast(g) = -x \cdot g$ by the remarks made in
Section \ref{xxsec1.5}, and $\varphi(\mathfrak{g}) = [x, \mathfrak{g}]$.
Using Proposition \ref{xxsec1.4} we can deduce that
\[\left(\prod_{l \in \mathbb{P}([x,\mathfrak{g}]) \setminus \mathbb{P}
(\ker \beta(y))} v_l^{p^s}\right) (x \cdot \beta(y))(Y) \in
XB.\] Swapping $x$ and $y$, we obtain
\[\left(\prod_{l \in \mathbb{P}([y,\mathfrak{g}]) \setminus \mathbb{P}
(\ker \beta(x))} v_l^{p^s}\right) (y \cdot \beta(x))(Y) \in
XB.\] Now $x \cdot \beta(y) = -y \cdot \beta(x)$ by Lemma
\ref{xxsec1.7}(a), and
\[\left(\mathbb{P}([x,\mathfrak{g}]) - \mathbb{P}(\ker \beta(y))\right)
\cap \left(\mathbb{P}([y, \mathfrak{g}]) - \mathbb{P}(\ker
\beta(x))\right) = \emptyset\] by Lemma \ref{xxsec1.7}(b). Hence the
two products occurring above are coprime, which allows us to deduce
that
\[(x \cdot \beta(y))(Y) \in XB\]
for all $x,y\in\mathfrak{g}$. Since $\mathfrak{g}^\ast$ generates
$\mathcal{D}$ as a $B$-module, it will be now enough to
show that $\{x \cdot \beta(y) : x, y\in \mathfrak{g}\}$ spans
$\mathfrak{g}^\ast$. But this is (L1).
\end{proof}

\subsection{Chevalley Lie algebras over $\Zp$}
\label{xxsec3.4} 
Let $\Phi$ be an indecomposable root system, let $C := \Phi(\Zp)$ 
be the Lie algebra over $\Zp$ constructed from a
Chevalley basis \cite[p.37]{CSM}, let 
$t \geq \epsilon$ and consider the powerful Lie algebra $L = p^t C$. 
Let $\mathfrak{g} = N/pN$ be the finite dimensional $\Fp$-Lie algebra constructed from $L$ in $\S\ref{xxsec3.1}$.

Recall that $p$ is a \emph{nice prime} for $\Phi$ if $p\geq 5$ 
and if $p \nmid n+1$ when $\Phi$ is the root system $A_n$.

\begin{thm} 
Retain the notation as above and suppose that $p$ is a 
nice prime for $\Phi$. Then
\begin{enumerate}[{(}a{)}]
\item 
$\Phi(\Fp)$ is a non-abelian simple $\Fp$-Lie algebra,
\item 
$N = C$ and $\mathfrak{g} = \Phi(\Fp)$,
\item 
$L$ satisfies $(L*)$.
\end{enumerate}
\end{thm}

\begin{proof}
(a) By construction, $\Phi(\Fp)$ is never abelian.
Under our assumptions on $p$, $\Phi(\Fp)$ is simple
\cite[p.181]{S}.

(b) Clearly $C \subseteq N$. Let $x \in N \backslash C$, for a
contradiction. Then we can find $k > 0$ such that $p^kx \in C
\backslash pC$. But now
\[[p^kx,C] \subseteq p^kC \subseteq pC\]
so $p^kx + pC$ is a non-zero central element of $C/pC = \Phi(\Fp)$. This
is a contradiction, because $\Phi(\Fp)$ is non-abelian simple by
part (a). Hence $N = C$ and $\mathfrak{g} = \Phi(\Fp)$.

(c) Let $\zeta : \mathfrak{g} \to V$ be defined by the obvious rule
\[\zeta(x + pN) = p^tx + pL.\]
This is clearly a $\mathfrak{g}$-module isomorphism.

Consider the \emph{normalized Killing form} on $\mathfrak{g}$,
defined by
\[(x + pN,y + pN) = \frac{\tr(\ad(x)\ad(y))}{2h}\]
for all $x,y\in N = \Phi(\Zp)$, where $h$ is the Coxeter number for
$\Phi$. This form is clearly $\mathfrak{g}$-invariant. By
\cite[Proposition 4]{GN} this form is non-zero and hence the radical
$\mathfrak{r} := \{x \in \mathfrak{g} : (x,\mathfrak{g}) = 0\}$ of
the form is a proper subspace of $\mathfrak{g}$. But $\mathfrak{r}$
is an ideal of $\mathfrak{g}$ and $\mathfrak{g}$ is simple, so
$\mathfrak{r} = 0$ and hence the form is non-degenerate.

Finally, $\mathfrak{g}$ is perfect because
$[\mathfrak{g},\mathfrak{g}]$ is an ideal of $\mathfrak{g}$, which
must be the whole of $\mathfrak{g}$ since $\mathfrak{g}$ is
non-abelian simple by part (a). The assertion follows from
the comments made before Theorem \ref{xxsec3.3}.
\end{proof}

\subsection{Proof of Theorem A}
\label{xxsec3.5}
\begin{lem} Let $L=L_1\oplus L_2$ where both $L_1$ and $L_2$
are powerful $\Zp$-Lie algebras. If $L_i$ satisfies condition
$(L*)$ for $i=1,2$, then so does $L$.
\end{lem} 

\begin{proof} Let $N, \mathfrak{g}$ and $V$ be defined as 
in Section \ref{xxsec3.1} for the Lie algebra $L$ (and 
similar terms for $L_1$ and $L_2$). It is clear that 
$N=N_1\oplus N_2$; consequently, $\mathfrak{g}=
\mathfrak{g}_1\oplus \mathfrak{g}_2$ and 
$V=V_1\oplus V_2$. The assertion now follows from 
the definition of $(L*)$.
\end{proof}

\begin{proof}[Proof of Theorem A]
Applying the lemma and Theorem \ref{xxsec3.4}, we see that $p^t\Phi(\Zp)$ satisfies $(L*)$. The result now follows
from Theorem \ref{xxsec3.3}.
\end{proof}

\section{Remarks on non-nice primes}
\subsection{} 
\label{xxsec4.1}
Suppose $\Phi=A_n$ and $p$ divides $n+1$, and let $G$ be a uniform group of type $\Phi$. Let $h_1, \ldots, h_n$ be the co-roots occurring in a Chevalley basis for the $\Zp$-Lie algebra $\Phi(\Zp) \cong \mathfrak{sl}_{n+1}(\Zp)$, and let 
\[z:=\sum_{i=1}^n ih_i \in \Phi(\Zp).\]
Then $\mathfrak{g} := \Phi(\Fp)$ has a one-dimensional centre generated by the image $\overline{z}$ of $z$ in $\mathfrak{g}$, and this fact causes the derivation hypothesis to fail for $(KG,KG^p)$.

However, a version of Corollary \ref{xxsec0.3} still holds.

\begin{thm}
Let $G$ be a uniform pro-$p$ group of type $A_n$. Then $\Omega_G$ has no non-trivial 
two-sided reflexive ideals.
\end{thm}

The proof is similar to the one given in \cite[Section 8]{AWZ} 
and needs the following lemma. 

\begin{lem} Let $L = p^t\Phi(\Zp)$ for some $t\geq 1$, let $L_1=pL + p^t \Zp z$ and let $L_2=pL$. Write $G=\exp(L)$, $G_1=\exp(L_1)$ and $G_2=\exp(L_2)$. Then
\begin{enumerate}[{(}a{)}]
\item
The Frobenius pair $(KG,KG_1)$ satisfies the derivation
hypothesis.
\item
The Frobenius pair $(KG_1,KG_2)$ satisfies the derivation
hypothesis.
\end{enumerate}
\end{lem}

\begin{proof}[Sketch of the proof] 
(a) Let $\mathfrak{f} := \Phi(\Fp)/\langle \overline{z}\rangle$. This is a simple Lie algebra \cite[p.181]{S} and there is an 
$\mathfrak{f}$-invariant non-degenerate bilinear form on $\mathfrak{f}$ \cite[6.4(b)]{J}. 
It induces an $\mathfrak{g}$-invariant bilinear form 
$(\;,\;)$ on $\mathfrak{g}$. Let $\beta : \mathfrak{g} \to \mathfrak{g}^\ast$ be the 
$\mathfrak{g}$-module homomorphism associated to $(\;,\;)$. Then the image of 
$\beta$ is equal to $\mathfrak{f}^*$, the annihilator of $\langle \overline{z}\rangle$ in $\mathfrak{g}^\ast$.

The proof of Theorem \ref{xxsec3.3} implies that if $Y$
lies in the $\mathbf{a}$-closure of $XB$ for all 
sources of derivations $\mathbf{a}$ of $(KG,KG_1)$, then $(B\otimes 
\mathfrak{g}\cdot\beta(\mathfrak{g}))(Y)\in XB$.
By the last paragraph, $\mathfrak{g}\cdot\beta(\mathfrak{g})
=\mathfrak{f}^*$. The assertion is proved by noting
that $\mathcal{D}:=\Der_{B_1}(B)$ is isomorphic 
to $B\otimes \mathfrak{f}^*$.

(b) This is an easier case than (a). Since $\Der_{B_2}(B_1)$
is isomorphic to $B_1\otimes K\overline{z}^*\cong B_1$, we can apply the argument
in the second half of the proof of \cite[Proposition 8.1]{AWZ}, after making the
appropriate changes. Therefore $(KG_1,KG_2)$ satisfies 
the derivation hypothesis.
\end{proof}

Using these techniques, we have verified that Corollary \ref{xxsec0.3} holds for all $(p, \Phi)$, except for $\{p=2, \Phi \in \{B_n, C_n, D_n, F_4, E_7\}\}$ and $\{p=3,\Phi\in \{G_2, E_6\}\}$. We believe that it holds in these exceptional cases as well.

\section*{Acknowledgments}
The authors would like to thank the referee for his helpful comments. K.Ardakov thanks the University of Sheffield for financial support. F. Wei is supported by a research fellowship from the China 
Scholarship Council and by the Department of Mathematics 
at the University of Washington. J.J. Zhang is supported by 
the US National Science Foundation and the Royalty Research 
Fund of the University of Washington.

\end{document}